\documentclass[10pt]{amsart}
\usepackage{amsmath,amssymb,amsthm}
\usepackage{graphicx}
\usepackage{color}
\newtheorem{theorem}{Theorem}    
\newtheorem{lemma}{Lemma}

\newtheorem{corollary}{Corollary}
\theoremstyle{definition}
\newtheorem{definition}[theorem]{Definition}

\newtheorem{remark}[theorem]{Remark}
\newtheorem*{remark*}{Remark}
\newcommand{\Z}{\mathbb{Z}}
\newcommand{\R}{\mathbb{R}}

\newcommand{\Crossing}{
\raisebox{-3mm}{
\begin{picture}(24,28)
\put(0,2){\vector(1,1){24}}
\put(24,2){\vector(-1,1){24}}
\end{picture} } 
}

\newcommand{\PCrossing}{
\raisebox{-3mm}{
\begin{picture}(24,28)
\put(0,2){\vector(1,1){24}}
\put(10,16){\vector(-1,1){10}}
\put(24,2){\line(-1,1){10}}
\end{picture} } 
}

\newcommand{\NCrossing}{
\raisebox{-3mm}{
\begin{picture}(24,28)
\put(0,2){\line(1,1){10}}
\put(24,2){\vector(-1,1){24}}
\put(14,16){\vector(1,1){10}}
\end{picture} } 
}
\newcommand{\Smooth}{
\raisebox{-3mm}{
\begin{picture}(24,28)
\qbezier(0,2)(14,14)(0,26)
\qbezier(24,2)(10,14)(24,26)
\put(0,26){\vector(-1,1){0}}
\put(24,26){\vector(1,1){0}}
\end{picture}} 
}

\title{A note on knot fertility II}
\author[T.Ito]{Tetsuya Ito}
\address{Department of Mathematics, Kyoto University, Kyoto 606-8502, JAPAN}
\email{tetitoh@math.kyoto-u.ac.jp}
 
\begin{document}

\begin{abstract}
A knot $K$ is called $(m,n)$-fertile if for every prime knot $K'$ whose crossing number is less than or equal to $m$, there exists an $n$-crossing diagram of $K$ such that one can get $K'$ from the diagram by changing its over-under information. We give an obstruction for knot to be $(m,n)$-fertile. As application, we prove the finiteness of $(c(K)+f,c(K)+p)$-fertile knots for all $f,p$. We also discuss the nubmer of Seiefrt circle and writhe of minimum crossing diagrams.
\end{abstract}

\maketitle

\section{Introduction}
A (knot) \emph{shadow} is an immersed circle in $\R^2$ or $S^{2}$ having only double point singularities. We view a shadow as a knot diagram without over-under information. For a given knot diagram $D$ we assign the shadow $S=S(D)$ by forgetting its over-under information. Conversely, we say that a diagram $D$ is \emph{supported} by a shadow $S$ if $S(D)=S$. A knot $K$ is \emph{supported} by a shadow $S$ if $S$ supports a diagram of $K$ i.e., if one can get the knot $K$ from the shadow $S$ by suitably assigning over-under information.

Throughout the paper, we always assume that a shadow or a diagram is oriented. However, by knots we always mean \emph{unoriented} knots so we do not distinguish a knot and its inverse orientation. Thus when we say that a diagram or a shadow supports a knot, we always ignore its orientation. Contrary, we always distinguish a knot and its mirror image. However in the following discussions this distinction plays no essential role, mainly because if a shadow $S$ supports a knot $K$, then $S$ also supports the mirror image of $K$\footnote{Nevertheless, at several points (especially when we use Lemma \ref{lem:key-lemma}) we need care for the mirror image since the maximum self-linking number is sensitive to taking mirror image}.

Let $c(K)$ be the minimum crossing number of a knot $K$. We denote the crossing number of a diagram $D$ or a shadow $S$ by $c(D)$ and $c(S)$, respectively.

\begin{definition}
A knot $K$ is \emph{fertile} if for every prime knot $K'$ with $c(K')<c(K)$, there exists a minimal crossing diagram $D$ of $K$ whose shadow $S(D)$ supports $K'$. 
\end{definition}

In \cite{CH+} it was observed that up to ten crossings, only $0_1, 3_1, 4_1, 5_2, 6_2, 6_3, 7_6$ are fertile, and asked whether there exists other fertile knots. They also introduced the following notions as a natural generalization of fertility.

\begin{definition}
A knot $K$ is \emph{$(m,n)$-fertile} if for every prime knot $K'$ with $c(K') \leq m$, there exists a shadow $S$ with $c(S)=n$ that supports both $K$ and $K'$. The \emph{fertility number} $F(K)$ is the maximum $m$ such that $K$ is $(m,c(K))$-fertile.
\end{definition}

This note is a continuation of \cite{it-fertile}, where we investigated non-existence of fertile knots with large crossing numbers. We give more strong constraints for a knot to be $(m,n)$-fertile.

The \emph{genus} $g(D)$ of a diagram $D$ is the genus of the Seifert surface obtained by Seifert's algorithm. The \emph{canonical genus} $g_c(K)$ of a knot $K$ is defined by 
\[ g_c(K) = \min\{ g(D) \: | \: D \mbox{ is a diagram of }K \}.\]

Our first results is the following obstruction for a knot to be $(m,n)$-fertile.

\begin{theorem}
\label{theorem:F(K)-bound}
If a knot $K$ is $(m,n)$-fertile, then $g_c(K) \leq n-m+1$. In particular, 
\[ F(K) \leq c(K)+ 1 - g_c(K) \leq c(K) + 1 - \frac{1}{2}\max \deg_z P_K(v,z)\]
where $P_K(v,z)$ is the HOMFLY polynomial of $K$.
\end{theorem}

In \cite{Ha} Hanaki proved the inequality $F(K) \leq \frac{2}{3}c(K) + \frac{4}{3}$ for \emph{alternating} knots using Tait flyping theorem \cite{mt}, and he asked whether this can be extended for all knots. 

By using the following quantities we generalizes Hanaki's inequality for general knots.

\begin{definition}
The \emph{canonical genus defect} of a knot $K$ is defined by 
\[ cgd(K)= \min \{g(D)-g_c(K) \: | \: D \mbox{ is a minimum crossing diagram of } K \}.\]
The \emph{Seifert circle variation} $scv(K)$ of a knot $K$ is defined by
\[ scv(K)= \max \{s(D)-s(D') \: | \: D \mbox{ and } D' \mbox{ are minimum crossing  diagrams of } K \}\]
Here $s(D)$ is the number of Seifert circles of a knot diagram $D$. 
\end{definition} 

\begin{theorem}\label{theorem:fertility-number}
For a knot $K$, $F(K) \leq \frac{2}{3}c(K) + \frac{4}{3} + \frac{1}{3}scv(K)+\frac{2}{3}cgd(K)$.
\end{theorem}

\begin{corollary}
\label{cor:scv}
Assume that $scv(K)=cgd(K)=0$. Then $F(K) \leq \frac{2}{3}c(K) + \frac{4}{3}$. In particular, $K$ is not fertile whenever $c(K)>7$.
\end{corollary}

\begin{remark}
\label{remark:scs}
Every minimum crossing diagram $D$ of an alternating knot $K$ is alternating (Tait conjecture I, proven in \cite{ka,mu1,th1}) and $g(D)=g_c(K)=g(K)$ holds for an alternating diagram $D$ of $K$ \cite{cr,mu}. Thus $scv(K)=cgd(K)=0$ for an alternating knot $K$ so Corollary \ref{cor:scv} includes Hanaki's result as its special case. 
We emphasize that Theorem \ref{theorem:fertility-number} not only generalizes Hanaki's result, but also simplifies the proof because we does not use Tait flyping theorem, the deepest property of alternating knots (although we use other Tait conjectures).
\end{remark}

\begin{remark}
At first glance, the assumption $scv(K)=cgd(K)=0$ of Corollary \ref{cor:scv} is hard to check. However there is a simpler characterization; $scv(K)=cgd(K)=0$ if and only if there exists a minimum crossing diagram $D$ of $K$ such that $s(D)$ is minimum among all minimum crossing diagrams of $K$ and that $g(D)=g_c(K)$. In particular, if a knot $K$ admits a minimum crossing diagram $D$ such that $g(D)=g_c(K)$ and $s(D)=b(K)$ where $b(K)$ is the braid index of $K$, then $scv(K)=cgd(K)=0$.

Thus there are many non-alternating knots that satisfy the assumption $scv(K)=cgd(K)=0$. For example, torus knots, or, more generally, the closure of a positive $n$-braid that contains the full twist (that guarantees that braid index is $n$ \cite{fw}) satisfies $scs(K)=cgd(K)=0$.
\end{remark}

The next results establishes the finiteness of fertile knots.
\begin{theorem}
\label{theorem:C(K)-bound}
If a knot $K$ is $(m,n)$-fertile, then 
\[ c(K) \leq (2n-2m+1)(3n-3m+4) \]
In particular, for a given $f \in \Z$ and $p\geq 0$, there are only finitely many knots which are $(c(K)+f,c(K)+p)$-fertile.
\end{theorem}

In particular, this means that even for the case $c(S)>c(K)$, a shadow $S$ supports $K$ imposes severe restrictions for knots to be supported by $S$.

As another application, we completely determine $(k,k)$-fertile knots.
This answers the question (v) in \cite[Section 7]{CH+} in the following stronger form.

\begin{corollary}
\label{cor:kk-fertile}
If a non-trivial knot $K$ is $(k,k)$-fertile for some $k$ then $K$ is either a trefoil, figure eight, or, $5_2$. Moreover,
\begin{itemize}
\item A non-trivial knot $K$ is $(k,k)$-fertile for all $k$ if and only if $K$ is the trefoil.
\item A non-trivial knot $K$ is $(2j,2j)$-fertile for all $j\geq 2$ if and only if $K$ is the figure-eight knot.
\item A non-trivial knot $K$ is $(2j,2j)$-fertile for all $j\geq 3$ if and only if $K$ is the figure-eight knot or $5_2$.
\end{itemize}
\end{corollary}

Finally we get the following improvement of our previous result \cite[Theorem 1.1]{it-fertile} that nearly solves the question (i) in \cite[Section 7]{CH+}.

\begin{corollary}
\label{cor:fertile}
Let $K$ be a fertile knot.
\begin{enumerate}
\item If $c(K)$ is even then $K$ is either $4_1$, $6_2$, or $6_3$.
\item If $c(K)$ is odd, then $c(K) \leq 21$.
Moreover, if $K$ is none of $3_1,5_2,7_6$, $K$ satisfies the following properties.
\begin{itemize} 
\item[(a)] $g(K)=g_c(K)=2$, $cgd(K)=0$, and $b(K)=4$. 
\item[(b)] $K$ has a minimum crossing diagram $D_M$ having $c(K)-3$ Seifert circles and has a minimum crossing diagram $D_m$ having exactly four Seifert circles. In particular, $c(K) \leq scv(K)+7$.
\end{itemize}
\end{enumerate}
\end{corollary}

\section{Proof of Theorems}

For a knot diagram $D$ or a shadow $S$, by resolving each crossing as \[ \PCrossing, \NCrossing, \Crossing \to \Smooth \]
we get a disjoint union of oriented circles which we call the \emph{Seifert circles}. Using the nubmer of Seifert circle $s(D)$ of a diagram $D$ (or $s(S)$ of a shadow $S$), the \emph{genus} of a diagram $D$ and a shadow $S$ are given by
\[ g(D) = \frac{1}{2}(1 -s(D) + c(D)) \mbox{ and } g(S)=\frac{1}{2}(1-s(S) + c(S)), \]
respectively. 

We begin with obvious observations.
\begin{lemma}
\label{lem:easy-observation}
If a shadow $S$ supports a knot $K$, then the following holds.
\begin{enumerate}
\item[(i)] $c(K)\leq c(S)$.
\item[(ii)] $g(K)\leq g_c(K)\leq g(S)$.
\item[(iii)] $b(K)\leq s(S)$.
\end{enumerate}
\end{lemma}
\begin{proof}(i) and (ii) are obvious. (iii) follows from the famous fact that the braid index $b(K) = \min\{s(D) \: | \: D\mbox{ is a diagram of } K \}$ \cite{Ya}.
\end{proof}

These trivial observations yield the following useful constraint.

\begin{lemma}
\label{lem:braid-genus}
If knots $K$ and $K'$ are supported by a shadow $S$, then
\[ b(K) \leq s(S) \leq c(S)+1-2g_c(K')\]
\end{lemma}

\begin{proof}$2g_c(K')-1 \leq 2g(S)-1=-s(S)+c(S) \leq -b(K)+c(S)$.
\end{proof}

Applying the lemma for the $(2,m)$ torus knot (when $m$ is odd) or the $(2,m-1)$ torus knot (when $m$ is even) we get the following constraint for a knot $K$ to be $(m,n)$ fertile.

\begin{lemma}
\label{lem:braid-index}
If $K$ is $(m,n)$-fertile, then 
\[ b(K) \leq
\begin{cases}
n-m+2 & (m:\mbox{odd}) \\
n-m+3 & (m:\mbox{even})
\end{cases} 
\]
\end{lemma}

To deduce further constraints, we use the \emph{self-linking number} of $D$ defined by
\[ sl(D)=-s(D)+c_+(D)-c_-(D)=-s(D)+w(D). \]
Here $c_{+}(D)$ (resp. $c_-(K)$) is the number of \emph{positive} (resp. \emph{negative}) crossings of $D$ and $w(D)=c_+(D)-c_-(D)$ is the \emph{writhe} of $D$.
The \emph{maximal self-linking number} of a knot $K$ is defined by 
\begin{align*}
\overline{sl}(K) &= \max\{sl(\mathcal{T})\: | \: \mathcal{T} \mbox{ is a transverse knot topologically isotopic to } K\}\\
&= \max\{sl(D)\: | \: D \mbox{ is a diagram of } K \}.
\end{align*}

Let $c_{+}(K)$ (resp. $c_-(K)$) be the minimum number of \emph{positive} (resp. \emph{negative}) crossings of diagrams of $K$.  

\begin{lemma}\label{lem:key-lemma}
If a shadow $S$ supports a knot $K$, then 
\[ 2g(S)-1 \leq \overline{sl}(K) + 2(c(S)-c_+(K)) \]
\end{lemma}
\begin{proof}
Let $D$ be a diagram of $K$ whose shadow is $S$. 
\begin{align*}
2g(S)-1 &= -s(S)+c(S) = -s(D)+c_+(D) - c_-(D) + 2c_-(D)\\
&= sl(D) +2 (c(D) - c_+(D))\\
& \leq \overline{sl}(K) + 2(c(S)-c_+(K))
\end{align*}
\end{proof}

We are ready to prove theorems stated in introduction.

\begin{proof}[Proof of Theorem \ref{theorem:F(K)-bound}]
Let $S$ be a shadow with $c(S)=n$ that supports both $K$ and the $m$ crossing twist knot $T_{m}$. 
Since $T_m$ is alternating, its reduced alternating diagram $D$ satisfies $c_{\pm}(D)=c_{\pm}(K)$ \cite{mu2,th3} so 
\[ c_+(T_m) =m, c_-(T_m)=0, \overline{sl}(T_m) = 1 \]
if $m$ is odd, and
\[ c_+(T_m) =m-2, c_-(T_m) =2, \overline{sl}(T_m) = -3\]
if $m$ is even. 
(Here the maximum self-linking number is determined by Morton-Franks-Williams inequality $\overline{sl}(K)\leq \min \deg_v P_K(v,z)-1$ \cite{mo,fw}).
Thus in both cases,
\[ -2c_+(T_m)+ \overline{sl}(T_m) = -2m+1.\]
Therefore we conclude
\begin{equation}
\label{eqn:cs}
2g_c(K)-1 \leq 2g(S)-1 \leq \overline{sl}(T_m) + 2(n -c_+(T_m) ) = 2n-2m+1.
\end{equation}
The latter inequality follows from Morton's inequality $2g_c(K) \geq \max \deg_z P_K(v,z)$ \cite{mo}.
\end{proof}

\begin{proof}[Proof of Theorem \ref{theorem:fertility-number}]
Let $D_M$ and $D_m$ be minimum crossing diagrams of $K$ such that $scv(K)=s(D_M)-s(D_m)$. Then it follows that $cgd(K)=g(D_M)-g_c(K)$.
By Theorem \ref{theorem:F(K)-bound}
\begin{align*}
c(K)-s(D_M) &= 2g(D_M)-1 = 2g_c(K)-1 +2cgd(K)\\
&\leq 2c(K)+1-2F(K)+2cgd(K).
\end{align*}
On the other hand, since there exists a shadow $S$ with $c(S)=c(K)$ that supports both $K$ and $(2,F(K))$ or $(2,F(K)-1)$ torus knot, by Lemma \ref{lem:braid-genus}
\[ s(D_m) \leq  s(S) \leq c(K)-F(K)+3. \]
Therfore we conclude that
\[ c(K)+s(D_m)-s(D_M) = c(K) -scv(K) \leq 3c(K)+ 4 - 3F(K)+2cgd(K). \]
\end{proof}

To get a bound of the crossing number, we use quantitative Birman-Menasco finiteness theorem that relates the crossing number, braid index and genus. This is a generalization \cite[Theorem 2]{it-fertile} ($b(K)=3$ case) which was the key ingredient in our previous work.

\begin{theorem}\cite[Theorem 1.2]{it-qBM}
\label{theorem:Q-BM}
\[ c(K) \leq \begin{cases} 
2g(K)+1 & b(K)=2,\\
\frac{5}{3}(2g(K)+2) & b(K)=3,\\
(2b(K)-5)(2g(K)+b(K)-1) & b(K)\geq 4.\end{cases}\]
\end{theorem}

\begin{proof}[Proof of Theorem \ref{theorem:C(K)-bound}]
If a knot $K$ is $(m,n)$-fertile then $b(K) \leq n-m+3$ by Lemma \ref{lem:braid-index} and $g(K) \leq g_c(K) \leq n-m+1$ by Theorem \ref{theorem:F(K)-bound}. Therefore $c(K) \leq (2n-2m+1)(3n-3m+4)$ by Theorem \ref{theorem:Q-BM}.
\end{proof}

\begin{proof}[Proof of Corollary \ref{cor:kk-fertile}]
If a non-trivial knot $K$ is $(k,k)$-fertile for some $k$, $g(K)=g_c(K) = 1$  by Theorem \ref{theorem:F(K)-bound} and $b(K) \leq 3$  by Lemma \ref{lem:braid-index}. This implies that $K$ is either the trefoil, figure eight, or $5_2$. The rest of the assertion follows from \cite{CH+} where they showed the `if' direction.
\end{proof}

\begin{proof}[Proof of Corollary \ref{cor:fertile}]
Assume that $K$ is a fertile knot other than $0_1$, $3_1$, $4_1$, $5_2$, $6_2$, $6_3$, $7_6$. Then we already know that $c(K)>10$.

If $c(K)$ is even and $K$ is fertile, then by Lemma $b(K)\leq 3$ and $g(K)\leq g_c(K) \leq 2$ so by Theorem \ref{theorem:Q-BM}, $c(K) \leq 10$. 

Thus in the following we assume that $c(K)$ is odd. By Lemma $b(K)\leq 4$ and $g(K)\leq g_c(K) \leq 2$ so $c(K) \leq 21$ so by Theorem \ref{theorem:Q-BM}.
If $b(K)\leq 3$, then by the same argument $c(K) \leq 10$.

Next we show $g(K)=2$. Since $g_c(K)=2$ implies $g(K)=2$ \cite{st2}, if $g(K)=1$ then $g_c(K)=1$. A knot $K$ has $g_c(K)=1$ if and only if $K$ is the three-strand odd pretzel knot $K=P(2p+1,2q+1,2r+1)$ \cite{st1}. From the HOMFLY polynomial and Morton-Franks-Williams inequality, we conclude that a knot $K$ with $g_c(K)=1$ and $b(K)=4$ is one of $6_1,7_2,7_4,9_{46}$. None of them are fertile.

Finally, if $S$ is a shadow with $c(S)=c(K)$ supporting both $K$ and $(2,c(K)-2)$ torus knot, Lemma \ref{lem:braid-genus} shows that $4 = b(K) \leq s(S)\leq 4$. Thus $K$ admits a minimum crossing diagram $D_m$ having exactly four Seifert circles. Similarly, if $S$ is a shadow with $c(S)=c(K)$ supporting both $K$ and $c(K)-1$ crossing twist knot, inequality (\ref{eqn:cs}) in the proof of Theorem \ref{theorem:F(K)-bound} shows that 
\[ 3= 2g_c(K)-1 \leq 2g(S)-1 \leq 3\]
so $g_c(K)=g(S)=2$. Thus $cgd(K)=0$ and $K$ admits a minimum crossing diagram $D_M$ having $c(K)-3$ Seifert circles.
\end{proof}

\section{Variations of writhe and the number of Seifert circles}

It had been thought that the writhe of a minimum crossing diagram is a knot invariant. Although this is true for alternating knots \cite{mu2} (or, more generally, adequate knots \cite{th3}), in general it is false, as a famous Perko pair demonstrates. 
Thus it is interesting to explore the \emph{writhe variation} defined by
\[ wv(K)= \max\{ \frac{1}{2}(w(D)-w(D')) \: |\:  D \mbox{ and } D' \mbox{ are minimum crossing  diagrams of } K \}.\]

Although we expect that both $wv(K)$ and $scv(K)$ are small compared with $c(K)$,  little is known for $wv(K)$ and $scv(K)$ for a general knot $K$.
We close the paper by pointing out the following estimate which is interesting in its own right. 

\begin{theorem}\label{theorem:scv}
For a knot $K$
\[ scv(K), \frac{1}{2}wv(K) \leq c(K)-2g_c(K)+1-b(K)\begin{cases}= 0 & b(K)=2 \\
\leq \frac{2}{5}c(K) & b(K)=3 \\ 
\leq \frac{2b(K)-6}{2b(K)-5}c(K) &  b(K)>3 
\end{cases}
\]
\end{theorem}

\begin{proof}[Proof of Theorem \ref{theorem:scv}]
First we prove the inequality for $scv(K)$. Let $D_M$ and $D_m$ be a minimum crossing diagram of $K$ taken so that $scv(K)=s(D_M)-s(D_m)$.
Then $2g_c(K)-1\leq -s(D_M) + c(K)$ and $b(K)\leq s(D_m)$ so
\begin{align*}
scv(K) = s(D_M)-s(D_m) \leq c(K)-2g_c(K)+1-b(K)
\end{align*}

Next we prove the inequality for $wv(K)$.
By generalized Jones' conjecture proven in \cite{dp,lm}, when $D$ and $D'$ are knot diagrams of $K$,
\[ |w(D)-w(D')| \leq (s(D)+s(D')-2b(K)). \]
Thus  
\[ wv(K) \leq 2(s(D_M)-b(K)) \leq 2(c(K)-2g_c(K)+1-b(K))\]
Since $c(K)-2g_c(K)+1-b(K) \leq c(K)-2g(K)+1-b(K)$, the latter inequality follows from Theorem \ref{theorem:Q-BM}.
\end{proof}

\section*{Acknowledgement}
The author is partially supported by JSPS KAKENHI Grant Numbers 19K03490, 21H04428.

\end{document}